\documentclass{amsart}
\usepackage{amsopn}
\usepackage{amssymb, amscd}

\newcommand{\nc}{\newcommand}

\nc{\vg}{\mathfrak{v} } \nc{\wg}{\mathfrak{w} } \nc{\zg}{\mathfrak{z} }
\nc{\ngo}{\mathfrak{n} } \nc{\kg}{\mathfrak{k} } \nc{\mg}{\mathfrak{m} }
\nc{\bg}{\mathfrak{b} } \nc{\ggo}{\mathfrak{g} } \nc{\ggob}{\overline{\mathfrak{g}}
} \nc{\sog}{\mathfrak{so} } \nc{\sug}{\mathfrak{su} } \nc{\spg}{\mathfrak{sp} }
\nc{\slg}{\mathfrak{sl} } \nc{\glg}{\mathfrak{gl} } \nc{\cg}{\mathfrak{c} }
\nc{\rg}{\mathfrak{r} } \nc{\hg}{\mathfrak{h} } \nc{\tg}{\mathfrak{t} }
\nc{\ug}{\mathfrak{u} } \nc{\dg}{\mathfrak{d} } \nc{\ag}{\mathfrak{a} }
\nc{\pg}{\mathfrak{p} } \nc{\sg}{\mathfrak{s} } \nc{\pca}{\mathcal{P}}
\nc{\nca}{\mathcal{N}} \nc{\lca}{\mathcal{L}} \nc{\oca}{\mathcal{O}}
\nc{\mca}{\mathcal{M}} \nc{\tca}{\mathcal{T}} \nc{\aca}{\mathcal{A}}
\nc{\cca}{\mathcal{C}} \nc{\gca}{\mathcal{G}} \nc{\sca}{\mathcal{S}}
\nc{\hca}{\mathcal{H}} \nc{\bca}{\mathcal{B}} \nc{\dca}{\mathcal{D}}
\nc{\val}{\operatorname{val}}

\nc{\vp}{\varphi} \nc{\ddt}{\tfrac{{\rm d}}{{\rm d}t}} \nc{\im}{\mathtt{i}}

\nc{\SO}{\mathrm{SO}} \nc{\Spe}{\mathrm{Sp}} \nc{\Sl}{\mathrm{SL}}
\nc{\SU}{\mathrm{SU}} \nc{\Or}{\mathrm{O}} \nc{\U}{\mathrm{U}} \nc{\Gl}{\mathrm{GL}}
\nc{\Se}{\mathrm{S}} \nc{\Cl}{\mathrm{Cl}} \nc{\Spein}{\mathrm{Spin}}
\nc{\Pin}{\mathrm{Pin}} \nc{\G}{\mathrm{GL}_n} \nc{\g}{\mathfrak{gl}_n}

\nc{\RR}{{\Bbb R}} \nc{\HH}{{\Bbb H}} \nc{\CC}{{\Bbb C}} \nc{\ZZ}{{\Bbb Z}}
\nc{\FF}{{\Bbb F}} \nc{\NN}{{\Bbb N}} \nc{\QQ}{{\Bbb Q}} \nc{\PP}{{\Bbb P}}

\nc{\vs}{\vspace{.2cm}} \nc{\vsp}{\vspace{1cm}} \nc{\ip}{\langle\cdot,\cdot\rangle}
\nc{\ipp}{(\cdot,\cdot)} \nc{\la}{\langle} \nc{\ra}{\rangle} \nc{\unm}{\tfrac{1}{2}}
\nc{\unc}{\tfrac{1}{4}} \nc{\und}{\tfrac{1}{16}} \nc{\no}{\vs\noindent}
\nc{\lam}{\Lambda^2(\RR^n)^*\otimes\RR^n} \nc{\tangz}{{\rm T}^{\rm Zar}}
\nc{\nor}{{\sf n}} \nc{\eigen}{(k_1<...<k_r;d_1,...,d_r)}
\nc{\eigencero}{(0<k_2<...<k_r;d_1,...,d_r)} \nc{\mum}{/\!\!/}
\nc{\kir}{/\!\!/\!\!/} \nc{\Ri}{\tfrac{4}{||\mu||^2}\Ric_{\mu}}
\nc{\ds}{\displaystyle} \nc{\ben}{\begin{enumerate}} \nc{\een}{\end{enumerate}}
\nc{\f}{\frac}

\nc{\He}{\operatorname{Hess}} \nc{\ad}{\operatorname{ad}}
\nc{\Ad}{\operatorname{Ad}} \nc{\rank}{\operatorname{rank}}
\nc{\Irr}{\operatorname{Irr}} \nc{\End}{\operatorname{End}}
\nc{\Aut}{\operatorname{Aut}} \nc{\Inn}{\operatorname{Inn}}
\nc{\Der}{\operatorname{Der}} \nc{\Ker}{\operatorname{Ker}}
\nc{\Iso}{\operatorname{I}} \nc{\Diff}{\operatorname{D}} \nc{\Lie}{\operatorname{L}}
\nc{\tr}{\operatorname{tr}} \nc{\dif}{\operatorname{d}}
\nc{\sen}{\operatorname{sen}} \nc{\modu}{\operatorname{mod}}
\nc{\Ric}{\operatorname{R}} \nc{\Ricg}{\operatorname{Ric^{\gamma}}}
\nc{\Ricci}{\operatorname{Ric}} \nc{\sym}{\operatorname{sym}}
\nc{\symac}{\operatorname{sym^{ac}}} \nc{\symc}{\operatorname{sym^{c}}}
\nc{\scalar}{\operatorname{sc}} \nc{\grad}{\operatorname{grad}}
\nc{\ricci}{\operatorname{ric}} \nc{\ricciac}{\operatorname{ric^{ac}}}
\nc{\riccic}{\operatorname{ric^{c}}} \nc{\riccig}{\operatorname{ric^{\gamma}}}
\nc{\Rin}{\operatorname{M}} \nc{\Le}{\operatorname{L}} \nc{\tang}{\operatorname{T}}
\nc{\level}{\operatorname{level}} \nc{\rad}{\operatorname{r}}
\nc{\abel}{\operatorname{ab}} \nc{\CH}{\operatorname{CH}}
\nc{\mcc}{\operatorname{mcc}} \nc{\Adj}{\operatorname{Adj}}
\nc{\Pf}{\operatorname{Pf}}

\theoremstyle{plain}
\newtheorem{theorem}{Theorem}[section]
\newtheorem{proposition}[theorem]{Proposition}

\newtheorem{lemma}[theorem]{Lemma}

\theoremstyle{definition}

\theoremstyle{remark}
\newtheorem{remark}[theorem]{Remark}

\title[A Curve of algebras which are not Einstein Nilradicals]{A Curve of nilpotent Lie algebras which are not Einstein Nilradicals}

\author{Cynthia Will}

\address{FaMAF and CIEM, Universidad Nacional de C\'ordoba, 5000 C\'ordoba, Argentina}
\email{cwill@mate.uncor.edu}

\thanks{
Supported by grants from CONICET, FONCyT and SeCyT (UNC)}

\begin{document}

\maketitle

\begin{abstract}
The only known examples of noncompact Einstein homogeneous spaces are standard solvmanifolds (special solvable Lie groups endowed with a left invariant metric), and according to a long standing conjecture, they might be all. The classification of Einstein solvmanifolds is equivalent to the one of {\it Einstein nilradicals}, i.e. nilpotent Lie algebras which are nilradicals of the Lie algebras of Einstein solvmanifolds.  Up to now, there have been found very few examples of $\NN$-graded nilpotent Lie algebras that can not be Einstein nilradicals. In particular, in each dimension, there are only finitely many known.  We exhibit in the present paper two curves of pairwise non-isomorphic $9$-dimensional $2$-step nilpotent Lie algebras which are not Einstein nilradicals.
\end{abstract}

\section{Introduction}\label{intro}

A Riemannian manifold is called Einstein if its Ricci tensor is a scalar multiple of the metric. The study of  homogeneous Einstein Riemannian manifolds splits in two very different parts, compact and noncompact cases, according to the sign of the Einstein constant. In both cases, existence and non existence results are hard to find.

In the noncompact homogeneous case, all the known examples are {\it solvmanifolds}, that is, a simply connected solvable Lie group $S$ endowed with a
left invariant Riemannian metric.  According to Alekvseevskii's conjecture (see \cite[7.57]{Bss}), they might be all. Moreover, if  $S$ is a solvmanifold with metric Lie algebra $(\sg, \ip)$ (i.e. $\sg$ is the Lie algebra of $S$ and $\ip$ is the inner product determined on $\sg$ by the metric),
consider in $\sg$ the orthogonal decomposition  $\sg = \ag \oplus \ngo,$ where  $\ngo=[\sg,\sg].$ Then the solvmanifold $(S,\ip)$ is called {\it standard} if $[\ag,\ag]=0$. It has been recently proved in \cite{standard} that any   Einstein solvmanifold is standard.

The standard case has been deeply studied by Heber \cite{Hbr}, who has proved many structural and uniqueness results. By results of Lauret \cite{critical} we know that the information to decide if a given solvmanifold is Einstein or not, can be found in its nilradical. A nilpotent Lie algebra $\ngo$ which can be the nilradical (i.e. the maximal
nilpotent ideal) of the Lie algebra of a standard Einstein solvmanifold is called an {\it Einstein nilradical}.  More precisely, $\ngo$ is an Einstein nilradical if and only if $\ngo$ admits an inner product $\ip$ satisfying
\begin{equation}\label{defEn}
\Ricci_{\ip}=cI+D, \qquad\mbox{for some}\; c\in\RR,\; D\in\Der(\ngo),
\end{equation}

\noindent where $\Ricci_{\ip}$ denotes the Ricci operator of the corresponding left invariant metric.  Such metrics are called {\it nilsolitons}, as they are Ricci soliton metrics (i.e. very distinguished solutions to the Ricci flow, see \cite{soliton}).

The classification of Einstein solvmanifolds turns to be equivalent to the one of Einstein nilradicals, which is still wide open.  For quite some time, the only known obstruction for a nilpotent Lie algebra to be an Einstein nilradical
was that it has to admit an $\NN$-{\it gradation} (see \cite{Hbr}). It is also known that up to dimension $6$ every nilpotent Lie algebra is an Einstein nilradical (see \cite{Wll}). Recently, in \cite{lw}, \cite{Nkl1} and \cite{Nkl3}, several examples of $\NN$-graded Lie
algebras which are not Einstein nilradicals have been provided. The smallest dimensional ones are, as expected, in dimension $7$. The first ones, in \cite{lw}, are filiform Lie algebras of dimension $7$ and certain $2$-step nilpotent Lie algebras given by graphs attaining any dimension $\geq 11$. In \cite{Nkl1} Nikolayevsky proved that very few of the free nilpotent Lie algebras are Einstein Nilradicals, and finally in \cite{Nkl3} he studies the case when the nilpotent Lie algebra admits a simple derivation, and finds several non Einstein nilradicals.

However, all these results in \cite{lw,Nkl1,Nkl3} only produce finitely many nilpotent Lie algebras that are not Einstein nilradicals in each dimension. The aim of this paper is to exhibit two curves of pairwise non-isomorphic $9$-dimensional two-step nilpotent Lie algebras which are not Einstein nilradicals, by using tools from geometric invariant theory and moment maps given in \cite{lw}. Both curves have $3$-dimensional center and the simplest one is defined for $t>1$ by
$$\begin{array}{lll}
\tilde{\mu}_t(e_5,e_4)=e_7, & \tilde{\mu}_t(e_1,e_6)=e_8, &  \tilde{\mu}_t(e_3,e_2)=e_9, \\
\tilde{\mu}_t(e_3,e_6)=t\;e_7, & \tilde{\mu}_t(e_5,e_2)=t\;e_8, &  \tilde{\mu}_t(e_1,e_4)=t\;e_9, \\
\tilde{\mu}_t(e_1,e_2)=e_7.
\end{array}
$$

\vs \noindent {\it Acknowledgements.} I would like to thank J. Lauret for his invaluable help.

\section{Preliminaries}\label{pre}

Let $S$ be a standard Einstein solvmanifold with $\sg=\ag\oplus\ngo$ as in Section 1.  For some distinguished element $H\in\ag$, the eigenvalues of $\ad{H}|_{\ngo}$ are all positive integers without a common divisor,
say $k_1<...<k_r$.  If $d_1,...,d_r$ denote the corresponding multiplicities, then
the tuple
$$
(k;d)=(k_1<...<k_r;d_1,...,d_r)
$$
is called the {\it eigenvalue type} of the Einstein solvmanifold $S$. In every dimension, only finitely many eigenvalue types occur. It turns out that $\RR H\oplus\ngo$ is
also an Einstein solvmanifold (with just the restriction of $\ip$ on it). It is then
enough to consider rank-one (i.e. $\dim{\ag}=1$) metric solvable Lie algebras as
every higher rank Einstein solvmanifold will correspond to a unique rank-one
Einstein solvmanifold and certain abelian subalgebra of derivations of $\ngo$
containing $\ad{H}$.  All the above results were proved in \cite{Hbr}.

We fix an inner product vector space
$$
(\sg=\RR H\oplus\RR^n,\ip),\qquad \la H,\RR^n\ra=0,\quad \la H,H\ra=1,
$$
such that the restriction $\ip|_{\RR^n\times\RR^n}$ is the canonical inner product
on $\RR^n$, which will also be denoted by $\ip$.

Note that the metric Lie algebra corresponding to any $(n+1)$-dimensional rank-one
solvmanifold, can be modeled on $(\sg=\RR H\oplus\ngo,\ip)$ for some nilpotent Lie
bracket $\mu$ on $\RR^n$ and some $D\in\Der(\mu)$, the space of derivations of
$(\RR^n,\mu)$.  Indeed, these data define a solvable Lie bracket $[\cdot,\cdot]$ on
$\sg$ by
\begin{equation}\label{solv}
[H,X]=DX,\qquad [X,Y]=\mu(X,Y), \qquad X,Y\in\RR^n,
\end{equation}

\noindent and the solvmanifold is then the simply connected Lie group $S$ with Lie
algebra $(\sg,[\cdot,\cdot])$ endowed with the left invariant Riemannian metric
determined by $\ip$.

For a given $\mu$, there exists a unique symmetric derivation
$D_{\mu}$ (possibly zero) so that the rank-one solvmanifold
$S_{\mu}$, defined by the data $\mu,D_{\mu}$ as in (\ref{solv}),
has a chance of being Einstein among all those metric
solvable extensions of $(\mu,\ip).$  In fact, if we denote by $N_{\mu}$ the simply connected Lie group with Lie algebra $(\RR^n,\mu),$ endowed with the left invariant Riemannian metric determined by $\ip$ and by $\Ric_{\mu}$ the Ricci operator of $N_{\mu}$, then $S_{\mu}$ is Einstein if and only if
\begin{equation}\label{der}
\Ric_{\mu}=c_{\mu}I+D_{\mu}
\end{equation}
\noindent where $c_{\mu}=-4||\mu||^2\tr{\Ric_{\mu}^2}$ (see \cite{soliton} and \cite{critical}).

Note that conversely, any $(n+1)$-dimensional rank-one Einstein solvmanifold is
isometric to $S_{\mu}$ for some nilpotent $\mu$. Thus
the set $\nca$ of all nilpotent Lie brackets on $\RR^n$ parameterizes a space of
$(n+1)$-dimensional rank-one solvmanifolds
$$
\{ S_{\mu}:\mu\in\nca\},
$$
containing all those which are Einstein in that dimension.

If we consider the vector space
$$
V=\lam=\{\mu:\RR^n\times\RR^n\longrightarrow\RR^n : \mu\; \mbox{bilinear and
skew-symmetric}\},
$$
then
$$
\nca=\{\mu\in V:\mu\;\mbox{satisfies Jacobi and is nilpotent}\}
$$
is an algebraic subset of $V$ as the Jacobi identity and the nilpotency condition
can both be written as zeroes of polynomial functions.  There is a natural action of
$\G:=\Gl_n(\RR)$ on $V$ given by
\begin{equation}\label{action}
g.\mu(X,Y)=g\mu(g^{-1}X,g^{-1}Y), \qquad X,Y\in\RR^n, \quad g\in\G,\quad \mu\in V.
\end{equation}

Note that $\nca$ is $\G$-invariant and Lie algebra isomorphism classes are precisely
the $\G$-orbits.

We have that two solvmanifolds $S_{\mu}$ and $S_{\lambda}$ with
$\mu,\lambda\in\nca$ are isometric if and only if there exists $g\in\Or(n)$ such
that $g.\mu=\lambda$ (see \cite[Proposition 4]{critical}), and the same is true for two nilmanifolds $N_{\mu}$ and $N_{\lambda}$ (see
\cite[Appendix]{minimal}), where $\Or(n)$ denotes the subgroup of $\G$ of orthogonal
matrices.  

The canonical inner product $\ip$ on $\RR^n$ defines an $\Or(n)$-invariant inner
product on $V$, denoted also by $\ip$, as follows:
\begin{equation}\label{innV}
\la\mu,\lambda\ra=\sum\limits_{ijk}\la\mu(e_i,e_j),e_k\ra\la\lambda(e_i,e_j),e_k\ra.
\end{equation}

Since any $\mu\in\nca$ is nilpotent,
the Ricci operator of $N_{\mu}$ denoted by $\Ric_{\mu}$ is given by (see \cite[7.38]{Bss}),
\begin{equation}\label{ricci}
\begin{array}{rl}
\la\Ric_{\mu}X,Y\ra=&-\unm\displaystyle{\sum\limits_{ij}}\la\mu(X,e_i),e_j\ra\la\mu(Y,e_i),e_j\ra \\
&+\unc\displaystyle{\sum\limits_{ij}}\la\mu(e_i,e_j),X\ra\la\mu(e_i,e_j),Y\ra,
\end{array}
\end{equation}

\noindent for all $X,Y\in\RR^n$.  We note that the scalar curvature of
$N_{\mu}$ is given by $\scalar(\mu)=\tr{\Ric_{\mu}}=-\unc||\mu||^2$. Formula (\ref{ricci}) can actually be used to define a symmetric operator $\Ric_{\mu}$ for any
$\mu\in V$.

A Cartan decomposition for the Lie algebra $\g$ of $\G$ is given by
$\g=\sog(n)\oplus\sym(n)$, that is, in skew-symmetric and symmetric matrices
respectively.  We use the standard $\Ad(\Or(n))$-invariant inner product on $\g$,
\begin{equation}\label{inng}
\la \alpha,\beta\ra=\tr{\alpha \beta^{\mathrm t}}, \qquad \alpha,\beta\in\g.
\end{equation}

The action of $\g$ on $V$ obtained by differentiation of (\ref{action}) is given by
\begin{equation}\label{actiong}
\pi(\alpha)\mu=\alpha\mu(\cdot,\cdot)-\mu(\alpha\cdot,\cdot)-\mu(\cdot,\alpha\cdot),
\qquad \alpha\in\g,\quad\mu\in V.
\end{equation}

Let $\tg$ denote the set of all diagonal $n\times n$ matrices.  If $\{
e_1',...,e_n'\}$ is the basis of $(\RR^n)^*$ dual to the canonical basis then
$$
\{ v_{ijk}=(e_i'\wedge e_j')\otimes e_k : 1\leq i<j\leq n, \; 1\leq k\leq n\}
$$
is a basis of weight vectors of $V$ for the action (\ref{action}), where $v_{ijk}$
is actually the bilinear form on $\RR^n$ defined by
$v_{ijk}(e_i,e_j)=-v_{ijk}(e_j,e_i)=e_k$ and zero otherwise.  The corresponding
weights $\alpha_{ij}^k\in\tg$, $i<j$, are given by
$$
\pi(\alpha)v_{ijk}=(a_k-a_i-a_j)v_{ijk}=\la\alpha,\alpha_{ij}^k\ra v_{ijk},
\quad\forall \; \alpha=\left[\begin{smallmatrix} a_1&&\\ &\ddots&\\ &&a_n
\end{smallmatrix}\right]\in\tg,
$$
where $\alpha_{ij}^k=E_{kk}-E_{ii}-E_{jj}$ and $\ip$ is the inner product defined in
(\ref{inng}).  As usual $E_{rs}$ denotes the matrix whose only nonzero coefficient
is $1$ in the entry $(r,s)$.

If for $\mu=\sum\mu_{ij}^kv_{ijk}\in\nca$ we fix an enumeration of the set
$\left\{\alpha_{ij}^k:\mu_{ij}^k\ne 0\right\}$, we can define the symmetric matrix
\begin{equation}\label{defU}
U=\left[\left\langle\alpha_{ij}^k,\alpha_{i'j'}^{k'}\right\rangle\right],
\end{equation}

\noindent and state the following result.

\begin{theorem}\cite[Theorem 1]{Pyn}\label{tracy}
Assume that $\mu\in\nca$ satisfies $\Ric_{\mu}\in\tg$.  Then $S_{\mu}$ is Einstein
if and only if
\begin{equation}\label{Ueq}
U\left[(\mu_{ij}^k)^2\right]=\nu [1], \qquad \nu\in\RR,
\end{equation}

\noindent where $\left[(\mu_{ij}^k)^2\right]$ is a column vector in the same order
used in { \rm(\ref{defU})},  and $[1]$ is the column vector
with all entries equal to $1$.
\end{theorem}

Another useful tool comes from the $\G$-invariant stratification for the vector space $V$ with certain frontier properties defined in \cite{standard}.  The strata
are parameterized by a finite set $\bca$ of diagonal $n\times n$ matrices,
$$
V\smallsetminus\{ 0\}=\bigcup\limits_{\beta\in\bca}\sca_{\beta} \qquad \mbox{(disjoint union)}.
$$
For each $\mu \in V$, let us denote by $\beta_\mu \in \tg$ the only minimal norm vector in the convex hull of $\{\alpha_{ij}^k:\mu_{ij}^k\ne 0\}$, and let us consider the Weyl chamber of $\g$ given by
$$
\tg^+=\left\{\left[\begin{smallmatrix} a_1&&\\ &\ddots&\\ &&a_n
\end{smallmatrix}\right]\in\tg:a_1\leq...\leq a_n\right\}.
$$

\begin{theorem}\label{util}\cite[3.1,3.3,3.4]{lw}
\
\begin{itemize}
\item[(i)] If $[c_{ij}^k]$ is any solution to $U[c_{ij}^k]=\nu [1]$,
    $\nu\in\RR$, such that $\sum c_{ij}^k=1$ and all $c_{ij}^k\geq 0$, then
    $\beta_{\mu}=\sum c_{ij}^k\alpha_{ij}^k$ and $\nu=||\beta_{\mu}||^2$.

\item[(ii)] For each convex linear combination $\beta_{\mu}=\sum c_{ij}^k\alpha_{ij}^k$, we define a finite
    set of $\lambda$'s in $V$ associated to $\mu$ by
    $\lambda=\sum\pm\sqrt{c_{ij}^k}v_{ijk}$.  If
    $\Ric_{\lambda}\in\tg$ and $\mu$ degenerates to $\lambda$ (i.e.
    $\lambda\in\overline{\G.\mu}$), then $\mu\in\sca_{\beta}$, for $\beta$ the only
    element in $\tg^+$ conjugate to $\beta_{\mu}$.

\item[(iii)] Let $\mu\in\nca$, $\mu\ne 0$. If $S_{\mu}$ is Einstein then $\mu\in\sca_{\beta}$ for $\beta$ the only element in $\tg^+$ conjugate to $\Ri$.  In such case, the eigenvalue type of $S_{\mu}$ is a positive scalar
multiple of $\beta+||\beta||^2I$.
\end{itemize}
\end{theorem}

Note that from this, if one knows which stratum $\sca_{\beta}$ $\mu$ belongs to (for example by using (i)), one has the eventual eigenvalue type of $S_{\mu}$, and from there one should be able to decide if $S_{\mu}$ is an Einstein solvmanifold or not by using for example Lie theory tools.

We will finally recall some results on $2$-step nilpotent Lie algebras we are going to use. We refer to \cite{ratform} and the references therein.  Consider an orthogonal decomposition $\RR^n=\vg \oplus \zg$, and let $\ngo_{\mu}=(\RR^n,\mu)$, $\mu \in \mathcal{N}$ be a $2$-step nilpotent Lie algebra such that $\mu(\vg,\vg)=\zg$. We will say that the {\it type} of $\ngo_{\mu}$ is $(n_1,n_2)$ if $\dim \vg=n_1$ and $\dim \zg=n_2$.

For each $w\in\zg$ consider the linear transformation $J_\mu(w):\vg\longrightarrow\vg$ defined by
\begin{equation}\label{jota}
\la J_\mu(w)v_1,v_2\ra=\la\mu(v_1,v_2),w\ra, \qquad\forall\; v_1,v_2\in\vg.
\end{equation}

Recall that $J_\mu(w)$ is skew symmetric with respect to $\ip$ and the map
$J_\mu:\zg\longrightarrow\sog(n_1,\RR)$ is linear.  If $n_1$ is even, one can define the {\it Pfaffian form} $f_\mu:\zg \mapsto \RR$ of the $2$-step nilpotent Lie algebra $\ngo_{\mu}$ by
$$
f_\mu(w)=\Pf(J_\mu(w)), \qquad w\in \zg,
$$
where $\Pf:\sog(\vg)\mapsto \RR$ is the {\it Pfaffian}, that is, the only polynomial
function on $\sog(\vg)$ satisfying $\Pf(B)^2=\det{B}$ for all $B\in\sog(n_1)$ and $\Pf(J)=1$ for some fixed $J\in\sog(\vg)$.  Recall that $f_{\mu}$ is a homogeneous polynomial on $n_2$ variables of degree $n_1/2$.

Let $\ngo_{\mu'}=(\RR^n,\mu')$ be another $2$-step nilpotent Lie algebra satisfying $\mu'(\vg,\vg)=\zg$.  Then $\ngo_{\mu}$ and $\ngo_{\mu'}$ are isomorphic if and only if there exist invertible $A_1:\vg\longrightarrow\vg$ and $A_2:\zg\longrightarrow\zg$ such that
\begin{equation}\label{iso}
A_1^tJ_{\mu'}(w)A_1=J_{\mu}(A_2^tw), \qquad\forall\; w\in\zg.
\end{equation}

It follows that if $\ngo_{\mu}$ and $\ngo_{\mu'}'$ are isomorphic, then their Pfaffian forms are {\it projectively equivalent}, that is, there exists $A\in\Gl_{n_2}(\RR)$ and $c\in \RR^*$
such that
$$
f_{\mu}(w)=cf_{\mu'}(Aw), \qquad\forall\; w\in\zg.
$$
In particular, for $2$-step nilpotent Lie algebras of type $(6,3)$, which is the case we are going to study in the present paper, the Pfaffian form is a homogeneous polynomial of degree $3$ in $3$ variables. Two of such polynomials $f,g$ are then projectively equivalent if there exist $A\in\Gl_3(\RR)$ and $c\in \RR^*$
such that
$$
f(x,y,z)=c\;g(A[x,y,z]^t).
$$
In that case, we note that
\begin{equation}\label{zeros}
AZ(f)=Z(g),
\end{equation}
where $Z(f):=\{(x,y,z):\in\RR^3:f(x,y,z)=0\}$.

\section{A curve of Einstein nilradicals}

We will introduce in this section a curve of Einstein nilradicals we are going to need later on. Let us consider the following family of $2$-step nilpotent Lie algebras. If $\{ e_1, \dots , e_9\}$ denote the canonical basis of $\RR^9$, for any $t \in \mathbb{R}$ let $\ngo_t=(\RR^9, \mu_t)$ be the Lie algebra with Lie bracket given by
\begin{equation}\label{fam2}
\begin{array}{lll}
\mu_t(e_5,e_4)=e_7, & \mu_t(e_1,e_6)=e_8, & \mu_t(e_3,e_2)=e_9,\\

\mu_t(e_3,e_6)=t\;e_7, & \mu_t(e_5,e_2)=t\;e_8, & \mu_t(e_1,e_4)=t\;e_9.
\end{array}
\end{equation}

It is easy to see that all of them are $2$-step nilpotent Lie algebras of type $(6,3),$
where $\{e_1, \dots, e_6\}$ and $\{e_7,e_8,e_9 \}$ are basis of $\vg$ and $\zg$ respectively, as in Section 2.  Recall that we will always have the canonical inner product on $\RR^9$ fixed.

\begin{lemma}\label{En} $\ngo_t$ is an Einstein nilradical for any $t \in \RR$.
\end{lemma}

\begin{proof}
By computing the Ricci operator, one can see that $(\RR^9,\mu_t)$ are Einstein nilradicals for any $t$. In fact, by (\ref{ricci}), straightforward calculation shows that $\Ric_{\mu_t}$ is given by
$$
\Ric_{\mu_t} -\frac{3+3t^2}{2} I+ \left[\begin{smallmatrix}
1+t^2&&&&&\\
&\ddots&&&&\\
&&1+t^2&&&\\
&&&2+2t^2&&\\
&&&&2+2t^2&\\
&&&&&2+2t^2
\end{smallmatrix}\right],
$$
and therefore, since the matrix on the right is easily seen to be a derivation of $\ngo_t$, we have that (\ref{defEn}) holds for $\ip$.
\end{proof}

Another way to prove this fact is by checking that $\mu_t$ satisfies the system given in (\ref{Ueq}), since in this case, with the ordering
$$\{\mu_{5,4}^7,\mu_{1,6}^8,\mu_{3,2}^9,\mu_{3,6}^7,\mu_{5,2}^8,\mu_{1,4}^9\},$$
$U$ is given by
$$
U:=\left[\begin{smallmatrix}
3&&&1&1&1\\
&3&&1&1&1\\
&&3&1&1&1\\
1&1&1&3&&\\
1&1&1&&3&\\
1&1&1&&&3
\end{smallmatrix}\right],
$$
\noindent and then $v=[1,1,1,t^2,t^2,t^2]^t$ satisfies $U.v = (3+3t^2) [1]$.  It then follows form this, by using Theorem \ref{util} (ii) (or alternatively (iii)), that $\mu_t \in S_\beta$ where
$$
\beta=\beta_{\mu_t}=\sum (\mu_t)_{ij}^k\alpha_{ij}^k
=\frac{1}{3}(-1,-1,-1,-1,-1,-1,1,1,1),
$$
and $S_{\mu_t}$ is an Einstein solvmanifold of eigenvalue type $(1<2;6,3)$ for all $t>0$.  

A straightforward computation shows that the Pfaffian form of $\ngo_t$ is given by
\begin{equation}\label{pffam2}
f_{\mu_t}(x,y,z)=(t+1)(t^2-t+1)xyz.
\end{equation}

We note that this family of Lie algebras is studied in \cite[Section 4]{GT}, where the base field is $\CC$, and with their notation these algebras are contained in their Family 2 and given by
$$
\mu_t=t.\mu_2+1.\mu_3.
$$
Therefore, as it is proved there, for $t,s\in\CC$, $t^3\ne\pm s^3$, $\mu_t$ and $\mu_{s}$ are non isomorphic as Lie algebras. Anyway, we will give an alternative proof of a little less stronger fact, that is enough for our purpose.

\begin{lemma}\label{isom} If $t,s \in (1,\infty)$ then $\ngo_t$ is isomorphic to $\ngo_s$ if and only if $t=s$.
\end{lemma}

\begin{proof}
Since $||\mu_t||^2 = 6(t^2+1)$, let us denote by $\lambda_t= \frac{\mu_t}{||\mu_t||}$, the normalization of $\mu_t$. It is not hard to see that
$$
\tr J_{\lambda_t}(w)^4 = \frac{1+t^4}{18(1+t^2)^2}(z^4+x^4+y^4)+\frac{4t^2}{18(1+t^2)^2}(x^2y^2+x^2z^2+y^2z^2),
$$
\noindent where $w=xe_7+ye_8+ze_9 \in \zg.$

Let us assume that $\ngo_t$ is isomorphic to $\ngo_s$, or equivalently, $(\RR^n,\lambda_t)$ is isomorphic to $(\RR^n,\lambda_s)$, $t,s\in (0,\infty)$.  Since both are Einstein nilradicals by Lemma \ref{En}, we know that
there is an orthogonal isomorphism between $\lambda_t$ and $\lambda_s$ (see \cite[Appendix]{minimal}).  Therefore, by (\ref{iso})
$$
A_1^{-1}J_{\lambda_s}(w)A_1=J_{\lambda_t}(A_2^{-1}w), \qquad\forall w \in \zg,
$$
for some $A_1 \in \Or(6)$ and $A_2 \in \Or(3)$.  From here, it is clear that
\begin{equation}\label{traza}
 \tr(J_{\lambda_s})(w)^4= \tr(J_{\lambda_t}(A_2^{-1}w))^4, \qquad \forall w\in\zg.
\end{equation}

Since $A_2  \in \Or(3)$ we have that $A_2(S)=S$, where $S$ is the sphere
$$
S=\{w=xe_7+ye_8+ze_9\in \zg: (x^2+y^2+z^2)=1\}.
$$
Therefore if we denote by $g_t(w)=\tr(J_{\lambda_t}(w))^4$ then $g_t(S)=g_s(S)$ by (\ref{traza}), and thus the maximum values of $g_t$ and $g_s$ restricted to $S$ must coincide.

Now, if $w=xe_7+ye_8+ze_9\in S$, it is easy to see that
$$
g_t(x,y,z) = \frac{t^2}{9(1+t^2)^2}+\frac{(1-t^2)^2}{18(1+t^2)^2}(x^4+y^4+z^4).
$$
Moreover, straightforward calculation shows that the maximum value of $g_t$ on $S$ is given by
$$M_t = \frac{t^2}{9(1+t^2)^2}+\frac{(1-t^2)^2}{18(1+t^2)^2},$$
and thus
$$
\frac{1+t^4}{18(1+t^2)^2}=\frac{1+s^4}{18(1+s^2)^2}.
$$
It is easy to see that $\frac{1+t^4}{18(1+t^2)^2}$ is an even  function which is injective in the interval $(1,\infty)$, which implies that $t=s$, as we wanted to show.
\end{proof}

\begin{remark} We are strongly using in the above proof the fact that these algebras are Einstein nilradicals for any $t \in \RR$, which may be seen as an application of Riemannian geometry to the classification problem of Lie algebras. Recall that there is no any evident invariant from Lie theory to distinguish these Lie algebras.
\end{remark}

\section{A curve of non Einstein Nilradicals}

Consider now a new family of nilpotent Lie algebras of the same type, coming from adding a new bracket to the family of Lie algebras $\ngo_t$ defined in (\ref{fam2}).
Let $\tilde{\ngo}_t$ be the Lie algebra given by $\tilde{\ngo}_t=(\RR^9, \tilde{\mu}_t)$ where
\begin{equation}\label{fam21}
\begin{array}{lll}
\tilde{\mu}_t(e_5,e_4)=e_7, & \tilde{\mu}_t(e_1,e_6)=e_8, &  \tilde{\mu}_t(e_3,e_2)=e_9, \\
\tilde{\mu}_t(e_3,e_6)=t\;e_7, & \tilde{\mu}_t(e_5,e_2)=t\;e_8, &  \tilde{\mu}_t(e_1,e_4)=t\;e_9, \\
\tilde{\mu}_t(e_1,e_2)=e_7.
\end{array}
\end{equation}

We begin by noting that for any $t$ these are also two-step nilpotent Lie algebras of type $(6,3)$.
In this case the Pfaffian form is easily seen to be
\begin{equation}\label{pffam21}
f_{\tilde{\mu}_t}(x,y,z)=tx^3-(t^3+1)xyz.
\end{equation}

\begin{lemma}\label{isomor} For any $t\in(1,\infty)$, $\ngo_t$ and $\tilde{\ngo}_t$ are not isomorphic as Lie algebras.
\end{lemma}

\begin{proof}
As we have said, if two nilpotent Lie algebras are isomorphic then its Pfaffian forms are projectively equivalent. In this situation, for each fixed $t$ straightforward calculation shows that $f_{\mu_t}$ and $f_{\tilde{\mu}_t}$ are projectively equivalent to $P=xyz,$ and $\tilde{P}=xyz+x^3,$ respectively.  Let $Z$ and $\tilde{Z}$ denote the set of zeros of $P$ and $\tilde{P}$, respectively. It is easy to see that $Z$ is the union of three planes in $\RR^3$ defined by $x=0$, $y=0$ and $z=0$. On the other hand, $\tilde{Z}$ is the union of the plane $x=0$ and the algebraic variety $x^2+yz=0$, and therefore one can not be the image of the other by a linear map. Hence, by (\ref{zeros}), the associated Lie algebras can not be isomorphic, concluding the proof (see \cite[Section 10.3]{D} for an alternative proof of the fact that these two polynomials are not projectively equivalent).
\end{proof}

The above lemma also follows from \cite{GT}. Indeed, with their notation, $\mu_t$ corresponds to Family $2$ and  $\tilde{\mu}_t$ corresponds to Family $2$ with nilpotent part 2 (see \cite[Section 4.1, Table 3]{GT}).

\begin{proposition}\label{noneinstein}  $\tilde{\ngo}_t$ is not an Einstein nilradical for any $t\in(1,\infty)$.
\end{proposition}

\begin{proof}
We begin by noting that in this case, by ordering the brackets as before with the new one at the end, we obtain that
$\tilde{U}$ is given by
$$
\tilde{U}:=\left[\begin{smallmatrix}
3&&&1&1&1&1\\
&3&&1&1&1&1\\
&&3&1&1&1&1\\
1&1&1&3&&&1\\
1&1&1&&3&&1\\
1&1&1&&&3&1\\
1&1&1&1&1&1&3
\end{smallmatrix}\right],
$$
and thus $[1,1,1,1,1,1,0]^t$ is a solution to the system (\ref{Ueq}). Therefore we can apply again Theorem \ref{util}, (i), and get that
$$
\beta_{\tilde{\mu}_t}=\sum v_{ij}^k\alpha_{ij}^k,
$$
\noindent that is, the diagonal matrix with entries $\frac{1}{3}(-1, \dots, -1,1,1,1).$

Note that $\beta_{\tilde{\mu}_t}$ coincides with $\beta_{\mu_t},$ which comes from the fact that $v_{1,2}^7=0.$
Moreover, we can see that for each $t$, $\tilde{\mu}_t$ degenerates to $\mu_t$ and therefore $\tilde{\mu}_t \in S_\beta$ where $\beta=\beta_{\mu_t}=\beta_{\tilde{\mu}_t}$ (see Theorem \ref{util} (ii)). In fact, it is not hard to check that if we define $A(s)=(-s,-s,s,s,0,0,s,-s,0)\in\tg$ then $\lim\limits_{s\rightarrow \infty} \phi_s.\tilde{\mu_t}=\mu_{t},$ where $\phi_s=e^{-A(s)}.$
Note that $A(s) \in \slg(6)\times \slg(3)$ for any $s.$
Assume now that $\tilde{\ngo_t}$ is an Einstein nilradical. Therefore, by Theorem \ref{util} (iii), its eigenvalue type is a multiple of $\beta + ||\beta||^2 I,$ which is the diagonal matrix with entries $\frac{2}{3}(1,\dots,1,2,2,2).$
That is, the eigenvalue type of $\tilde{\ngo}_t$ must be $(1<2;6,3),$ or equivalently,
${\Ric_{\tilde{\mu}}}|_\vg$ and ${\Ric_{\tilde{\mu}}}|_\zg$ are both multiple of the identity. Therefore, by \cite[Proposition 9.1]{degeneration}, the $\Sl(6)\times\Sl(3)$ orbit of $\tilde{\ngo}_t$ is closed for each $t$. This implies that $\mu_t$ and $\tilde{\mu}_t$ should be isomorphic Lie algebras as we have seen that $\mu_t\in\overline{\Sl(6)\times\Sl(6).\tilde{\mu}_t}$, which is a contradiction by Lemma \ref{isomor}.  From all this we can deduce that $\tilde {\ngo}_t$ is not an Einstein nilradical for any $t$.
\end{proof}

We will finally prove the fact that our family $\tilde{\ngo}_t$ of Lie algebras is really a curve in the moduli space of isomorphism classes. This can be proved by using the classification given in \cite{GT}, but according with what we have been doing, we will give a self contained proof of this fact.

\begin{theorem} \label{nonisom} $\tilde{\ngo}_{t},$  $t\in (1,\infty)$, is a curve of pairwise non-isomorphic nilpotent Lie algebras, none of which is an Einstein nilradical.
\end{theorem}

\begin{proof} We have already seen that $\tilde{\mu}_{t}$ are not Einstein nilradicals, so it is enough to show that they are non isomorphic as Lie algebras. Let us take $t,s \in (1,\infty)$, and assume that $\tilde{\mu}_t$ and $\tilde{\mu}_{s}$ are isomorphic.  Since they are type $(6,3)$ nilpotent Lie algebras, it can be seen that
$$
\Sl(6)\times\Sl(3).\tilde{\mu}_t=\Sl(6)\times\Sl(3).(c\tilde{\mu}_s), \quad\mbox{for some}\; c>0.
$$
On the other hand, as we have already seen, $\mu_t \in \overline{\Sl(6)\times\Sl(3).\tilde{\mu}_t}$, for any $t\in \RR$, and therefore, $\mu_t$ and $c\mu_s$ belong to the closure of a single $\Sl(6)\times\Sl(3)$-orbit.  Now, since $\mu_t$ is an Einstein nilradical for any $t$, with eigenvalue type $(1<2;6,3)$, as we have already pointed out,  $\Sl(6)\times\Sl(3).\mu_t$ is closed for any $t \in \RR$ (see \cite[Proposition 9.1]{degeneration}).  But by \cite{RS} there is only one closed orbit in the closure of an orbit of any reductive algebraic group linear action. Hence,
$$
\Sl(6)\times\Sl(3).\mu_t = \Sl(6)\times\Sl(3).(c\mu_s),
$$
which implies that $\mu_t$ and $\mu_{s}$ are isomorphic.  Thus $t=s$ by Lemma \ref{isom}, concluding the proof.
\end{proof}

\section{Another curve.}

In this section we consider a new curve arising from adding one more bracket to family (\ref{fam21}). We will show that this is a curve of nilpotent Lie algebras none of which are Einstein nilradicals. Since we are following the same lines as in the previous section, we shall only outline some of the proofs. The main result in this case is that this curve is not isomorphic to the previous ones.

Let $\overline{\mu}_t$ be the two-step nilpotent Lie algebras of type $(6,3)$ given by $\overline{\ngo}_t=(\RR^9, \overline{\mu}_t)$ where
\begin{equation}\label{fam31}
\begin{array}{lll}
\overline{\mu}_t(e_5,e_4)=e_7 & \overline{\mu}_t(e_1,e_6)=e_8 & \overline{\mu}_t(e_3,e_2)=e_9\\

\overline{\mu}_t(e_3,e_6)=t\;e_7 & \overline{\mu}_t(e_5,e_2)=t\;e_8 & \overline{\mu}_t(e_1,e_4)=t\;e_9\\

\overline{\mu}_t(e_1,e_2)=e_7 & \overline{\mu}_t(e_3,e_4)=e_8.&
\end{array}
\end{equation}

These algebras are also considered in \cite{GT}, included in their Family 2 with nilpotent part 1 (see \cite[Table 3]{GT}).

In this case the Pfaffian form is given by
\begin{equation}\label{pffam31}
f_{\overline{n}_t}(x,y,z)=t(x^3+y^3)-(t^3+1)xyz.
\end{equation}

\begin{lemma} For any $t\in (1,\infty)$, $\overline{\ngo}_t$ is not isomorphic to any of $\ngo_s$ or $\tilde{\ngo}_s$, $s\in (1,\infty)$.
\end{lemma}

\begin{proof}
For each fixed $t\in (1,\infty)$, the Pfaffian form $f_{\overline{\mu}_t}$ is projectively equivalent to
$$
\overline{P}=xyz+x^3+y^3.
$$
It is easy to see that the set of zeros $\overline{Z}$ of $\overline{P}$ does not contain any $2$-dimensional subspace, and hence there is no any linear transformation taking $\overline{Z}$ to the set of zeros of $P$ or $\tilde{P}$ (see the proof of Lemma \ref{isomor}).  This implies that $\overline{\ngo}_t$ can not be isomorphic to either $\ngo_s$ or $\tilde{\ngo}_s$ as we wanted to show (see also \cite{D} pp. 159).
\end{proof}

As in the previous case, to see if  $\overline{\ngo}_t$ is an Einstein nilradical or not, we note that
by ordering the brackets as before with the new one at the end, we obtain that
$\overline{U}$ is given by
$$
\overline{U}:=
\left[\begin{smallmatrix}
3&&&1&1&1&1&1\\
&3&&1&1&1&1&1\\
&&3&1&1&1&1&1\\
1&1&1&3&&&1&1\\
1&1&1&&3&&1&1\\
1&1&1&&&3&1&1\\
1&1&1&1&1&1&3&\\
1&1&1&1&1&1&&3
\end{smallmatrix}\right],
$$
and so a solution to (\ref{Ueq}) is given by $(1,1,1,1,1,1,0,0)$.  Using the same arguments as before, we get that $\beta_{\overline{\mu}_t}=\frac{1}{3}(-1, \dots, -1,1,1,1)$.  Analogously, we can see that for each $t$, $\overline{\mu}_t$ degenerates to $\mu_t$ and therefore $\overline{\mu}_t \in S_\beta$ where $\beta$ as defined above (see Theorem \ref{util}, (ii)). Thus $\overline{\mu}_t$ and $\tilde{\mu}_s$ belongs to the same stratum for any $t,s \in (1,\infty)$.  In this case the degeneration can be realized by
$\phi_s=e^{-A(s)},$ where $A(s)=(0,-s,0,-s,-3s,2s,2s,2s,-s)\in\tg$.  Therefore, since $A(s) \in \slg(6)\times \slg(3)$ for any $s$, we can see that $\overline{\ngo}_t$ is not an Einstein nilradical for all $t,$  by arguing exactly as in the proof of Lemma \ref{noneinstein}.
Finally, we can see that these algebras give rise a curve by using \cite{GT}, or by following the lines of the proof of Theorem \ref{nonisom}.

\end{document}